\documentclass[reqno,11pt]{article}
\usepackage[a4paper,margin=3cm, top=2.5cm, bottom=2.5cm]{geometry}
\usepackage{amsmath,amsthm,amssymb}

\newcommand{\R}{\mathbb{R}}
\newcommand{\N}{\mathbb{N}}
\renewcommand{\Pr}{\mathbb{P}}
\newcommand{\II}{\mathbf{1}}
\newcommand{\E}{\mathbb{E}}
\newcommand{\Var}{\textup{Var}}
\newcommand{\Vol}{\textup{Vol}}

\newtheorem{theorem}{Theorem}
\newtheorem{lemma}[theorem]{Lemma}

\newtheorem{prop}[theorem]{Proposition}
\newtheorem{corollary}[theorem]{Corollary}

\theoremstyle{definition}

\newtheorem{remark}[theorem]{Remark}

\numberwithin{theorem}{section}
\numberwithin{equation}{section}

\begin{document}
\title{Volume properties of high-dimensional Orlicz balls}
\author{F. Barthe and P. Wolff}
\maketitle

\begin{abstract}
    We prove asymptotic estimates for the volume of families of Orlicz balls in high dimensions. As an application, we describe a large family of Orlicz balls which verify a famous conjecture of Kannan, Lov\'asz and Simonovits about spectral gaps. We also study the asymptotic independence of coordinates on uniform random vectors on Orlicz balls, as well as integrability properties of their linear functionals.
    
\end{abstract}

Lebesgue spaces play a central role in functional analysis, and enjoy remarkable structural properties. A natural extension of this family is given by the class of Orlicz spaces, which also enjoy a wealth of remarkable properties, see e.g. \cite{RAO-REN}. Similarly, for $p\ge 1$, the unit balls of $\mathbb R^n$ equipped  with the $\ell_p$-norm, often denoted by $B_p^n=\{x\in \mathbb R^n;\; \sum_i |x_i|^p \le 1\}$, are well studied convex bodies, and usually the first family of test cases for new conjectures. Their simple analytic description allows for many explicit calculations, for instance of their volume. A simple probabilistic representation of uniform random vectors on $B_p^n$, given in terms of i.i.d. random variables of law  $\exp(-|t|^p)\, dt/K_p$ is available, see \cite{bgmn}. It allows to investigate various fine properties of the volume distribution on $B_p^n$. The study of general Orlicz balls is more difficult, due to the lack of explicit formulas, in particular for the volume of the set itself. 
 
In this note, we show that probabilistic methods allow to derive precise asymptotic estimates of the volume of Orlicz balls when the dimension tends to infinity, and rough estimates which are valid in every dimension. This allows us to complement a result of Kolesnikov and Milman \cite{kolesnikov-milman} on the spectral gap of uniform measures on Orlicz balls, by giving an explicit description of the range of parameters where their result applies, see Section \ref{sec:gap}. In Section \ref{sec:indep}, we show, among other results, the asymptotic independence of a fixed set of coordinates of uniform random vectors on some families of Orlicz balls of increasing dimensions. This is a natural extension of a classical observation (going back to Maxwell) about uniform vectors on Euclidean spheres and balls. The last section deals with properties of linear functionals of random vectors on Orlicz balls. 

After this research work was completed, we learned by J. Prochno of his independent work \cite{kabluchko-prochno} with Z. Kabluchko about similar volume asymptotics for Orlicz balls. Their paper uses sophisticated methods from the theory of large deviations, which have the potential to give more precise results for a given sequence of balls in increasing dimensions. Our approach is more elementary and focuses on uniform convergence over some wide range of parameters, as required by our applications to the spectral gap conjecture.

\section{Notation and statement}
Throughout this paper, a Young function is a non-negative convex function on $\mathbb R$ which vanishes only at 0. Note that we do not assume symmetry at this stage.
For a given Young function $\Psi \colon \R \to \R^+$, denote
\[
  B_\Psi^n = \Big\{ x \in \R^n \colon \sum_{i=1}^n \Psi(x_i) \le 1 \Big\}
\]
the corresponding $n$-dimensional Orlicz ball. Our aim is to estimate the asymptotic
volume of $B^n_{\Psi/E_n}= \{ x \in \R^n \colon \sum_{i=1}^n \Psi(x_i) \le E_n \}$
for relevant sequences $E_n$ of linear order in the dimension.

Let $\lambda>0$. Consider the following probability measure on $\R$,
\[\mu_\lambda(dt) = e^{-\lambda \Psi(t)} \frac{dt}{Z_\lambda},\] 
with $Z_\lambda$ being a normalization constant.
Let $X$ be a random variable  with the distribution $\mu_\lambda$. Set 
\[m = m_\lambda = \E \Psi(X),\qquad \sigma^2 = \sigma^2_\lambda = \Var\big( \Psi(X)\big).\]

Our aim  is to prove
\begin{theorem}\label{th:volume-asymptotic}
Consider a  Young function $\Psi$ and $\lambda > 0$.
Let $n\ge 1$ be an integer and $\alpha \in \mathbb R$. Set 
\[ E:=m_\lambda n +\alpha \sigma_\lambda\sqrt n,\]
then
\[ \begin{split}
  \Vol\big(B^n_{\Psi/E}\big) &= (Z_\lambda e^{\lambda m_\lambda})^n \frac{1}{\lambda \sigma_\lambda \sqrt{2\pi n}} \, e^{-\alpha^2/2} e^{\lambda \sigma_\lambda \sqrt{n} \alpha} (1 + O(n^{-1/2})) \\
 &= \frac{Z_\lambda^n e^{\lambda E}}{\lambda \sigma_\lambda \sqrt{2\pi n}} \, e^{-
\alpha^2/2} (1 + O(n^{-1/2})),
\end{split} \] 
where the term $O(n^{-1/2})$ depends on $\lambda, \Psi$ and, non-decreasingly in $|\alpha|$.
\end{theorem}

\begin{corollary} \label{cor:volume-asymptotic}
Consider a  Young function $\Psi$ and $\lambda > 0$.
Let  $(a_n)_{n\ge 1} $ be a bounded sequence, and  
$ E_n:=m_\lambda n +a_n \sqrt n,$
Then when the dimension $n$ tends to $\infty$,
\[ 
  \Vol\big(B^n_{\Psi/E_n}\big) \sim \frac{Z_\lambda^n e^{\lambda E_n}}{\lambda \sigma_\lambda \sqrt{2\pi n}} \, e^{-a_n^2/(2\sigma_\lambda^2)} .
\] 

\end{corollary}

Let us mention that the above results can be applied to $B^n_{\Psi/E_n}$ when $E_n=mn+a_n\sqrt n$ where $m>0$ is fixed 
and $(a_n)_n$ is a bounded sequence. Indeed  the next lemma ensures the existence of a  $\lambda>0$ such that $m=m_\lambda$.
\begin{lemma}
Let $\Psi$ as above. Then the map defined $(0,+\infty)$ to $(0,+\infty)$ by 
\[ \lambda\mapsto R(\lambda):=\frac{\int \Psi(t) e^{-\lambda \Psi(t)}dt}{\int e^{-\lambda \Psi(t)}dt}\] 
 is onto.
\end{lemma}
\begin{proof}
By hypothesis, $\int \exp(-\lambda \Psi)<\infty$ for all $\lambda>0$. This fact allows us to apply the dominated convergence theorem, and to show that the ratio $R(\lambda)$ is a continuous function of $\lambda>0$. Let us show that $\lim_{\lambda\to 0^+}R(\lambda)=\infty$ and  $\lim_{\lambda\to \infty}R(\lambda)=0$. The claim will then follow by continuity. 

Consider an arbitrary $K>0$. Since $\Psi\ge 0$,
\begin{align*}
\frac{\int \Psi e^{-\lambda\Psi}}{\int e^{-\lambda\Psi}}&\ge  \frac{K\int_{\Psi\ge K} e^{-\lambda\Psi}}{\int e^{-\lambda\Psi}}=K\left(1- \frac{\int_{\Psi< K} e^{-\lambda\Psi}}{\int e^{-\lambda\Psi}} \right)
\ge K\left(1- \frac{\mathrm{Vol}(\{x; \Psi(x)<K\})}{\int e^{-\lambda\Psi}} \right).
\end{align*}
By monotone convergence, $\lim_{\lambda\to 0^+} \int e^{-\lambda \Psi}=\infty$. Hence, $\liminf_{\lambda\to 0^+} R(\lambda)\ge K$. Since this holds for every
$K>0$, we conclude that $\lim_{\lambda\to 0^+} R(\lambda)=\infty$.

Let $\varepsilon>0$. As above, since $\Psi\ge0$,
\[\frac{\int \Psi e^{-\lambda\Psi}}{\int e^{-\lambda\Psi}} \le \varepsilon+ \frac{\int_{\Psi>\varepsilon} \Psi e^{-\lambda\Psi}}{\int e^{-\lambda\Psi}} . \]
Next, using $x\le e^x$ and for $\lambda>2$,
\[ \int_{\Psi>\varepsilon} \Psi e^{-\lambda\Psi} \le \int_{\Psi>\varepsilon}  e^{-(\lambda-1)\Psi}
= \int_{\Psi>\varepsilon}  e^{-\Psi} e^{-(\lambda-2)\Psi} \le e^{-(\lambda-2)\varepsilon} \int e^{-\Psi},\]
and 
\[ \int e^{-\lambda \Psi} \ge  \int_{\Psi\le \varepsilon/2} e^{-\lambda \Psi} \ge 
 e^{-\lambda \varepsilon/2} \mathrm{Vol}(\{x; \Psi(x)\le \varepsilon/2\}).\]
 Since $\Psi(0)=0$ and $\Psi$ is continuous, the latter quantity is positive. 
 Combining the above three estimates, we get 
 \[\frac{\int \Psi e^{-\lambda\Psi}}{\int e^{-\lambda\Psi}} \le \varepsilon+e^{-(\frac{\lambda}{2}-2) \varepsilon } \frac{\int e^{-\Psi}}{ \mathrm{Vol}(\{x; \Psi(x)\le \varepsilon/2\}) }. \]
 Letting $\lambda\to \infty$ yields $\limsup_{\lambda\to\infty} R(\lambda)\le \varepsilon$, for all $\varepsilon>0$.
\end{proof}

\section{Probabilistic formulation}
We start with a formula relating the volume with an expectation expressed in terms of independent
random variables.
Let $\lambda>0$.
Let $(X_i)_{i\in \N^*}$ be i.i.d. r.v.'s with the distribution $\mu_\lambda(dt) = e^{-\lambda \Psi(t)} \,dt/Z_\lambda$.  Recall that 
$m_\lambda = \E \Psi(X_i)$ and $\sigma^2_\lambda = \Var\big(\Psi( X_i)\big).$
We denote by $S_n$ the normalized central limit sums:
   \[S_n = \frac{1}{\sigma_\lambda \sqrt{n}} \sum_{i=1}^n (\Psi(X_i) - m_\lambda).\]
With this notation, we get the following representation for any $\lambda>0$
\begin{align}
  \Vol\big(B^n_{\Psi/E}\big) & = \int \II_{\{\sum_{i=1}^n \Psi(x_i) \le E\}} dx
  =\int \II_{\{\sum_{i=1}^n \Psi(x_i) \le E\}} Z_\lambda^n e^{\lambda \sum_{i=1}^n \Psi(x_i)}  \prod_{i=1}^n \mu_\lambda(dx_i) \nonumber\\
  &=Z_\lambda^n \E\left( e^{\lambda \sum_{i=1}^n \Psi(X_i)} \II_{\{\sum_{i=1}^n \Psi(X_i) \le E\}}\right)\nonumber  \\
  &= (Z_\lambda e^{\lambda m_\lambda})^n \E \left(e^{\lambda \sigma_\lambda \sqrt{n} S_n} \II_{\{S_n \le \frac{E - m_\lambda n}{\sigma_\lambda \sqrt{n}}\}}\right).
  \label{eq:probabilistic-volume}
\end{align}
By the Central Limit Theorem, $S_n$ converges in distribution to a standard Gaussian random variable when $n$ tends to infinity.
Such Gaussian approximation results allow to estimate the asymptotic behaviour of the above expectations. Nevertheless,
a direct application of the CLT or the Berry-Esseen bounds does not seem to be sufficient for our purposes.
A more 
refined analysis is required, 
 built on classical results and techniques on the distribution of sums of independent random variables which go back to Cram\'er~\cite{cramer-1938} (see also~\cite{bahadur-rao-1960}).  
 
 Theorem \ref{th:volume-asymptotic} is a direct consequence of the following one, applied to $Y_i=(\Psi(X_i)-m_\lambda)/\sigma_\lambda$. For a random variable $V$, let $\mathbb{P}_V$ and $\varphi_V$ denote the distribution and the characteristic function.
 
 \begin{theorem}\label{th:CLT-exp}
 Let $(Y_i)_{i\ge 1}$ be a sequence of i.i.d. real random variables such that $\mathbb E|Y_i|^3<\infty$, $\mathbb E Y_i=0$ and $\mathrm{Var}(Y_i)=1$.    Suppose $\varepsilon, \delta > 0$ are such that so-called Cram\'er's condition is satisfied for $Y_i$:
\begin{equation}\label{eq:cramer}
  |\varphi_{Y_i}(t)| \le 1 - \varepsilon \quad \textup{for $|t| > \delta$}.
\end{equation}
For $n\ge 1$, let  $S_n=(Y_1+\cdots+Y_n)/\sqrt{n}$. Then for $\ell >0$ and $\alpha\in \R$,
\[\mathbb E \left( e^{\ell \sqrt{n}\, S_n} \mathbf{1}_{S_n\le \alpha}\right) =\frac{1}{\ell \sqrt{2\pi n}} e^{\ell \sqrt{n}\alpha-\alpha^2/2} \big(1+O(n^{-1/2})\big). \]

\end{theorem} 
\begin{remark}
The term $O(n^{-1/2})$ involves an implicit dependence in $\ell, \alpha$ and the law of $Y_1$.
For $n\ge 16 \ell^2+ (2|\alpha|+1)^2 \ell^{-2}$, our argument provides a term $O(n^{-1/2})$ which depends only on
$(\ell, \delta, 1/\varepsilon, \nu_3:=\E |Y_i|^3, |\alpha|)$. Moreover the dependence is continuous in the parameters, and non-decreasing in all the parameters but $\ell$. This allows for uniform bounds when the parameters are in compact subsets of their domain. 
\end{remark}
\begin{remark}
Note that non-trivial $\varepsilon$ and $\delta$ exist by the Riemann-Lebesgue lemma as soon has the law of $Y_i$ is absolutely continuous.
 \end{remark}
 
 \section{Probabilistic preliminaries}
We start with  some useful lemmas.
The first one is a key estimate for quantitative central limit theorems, quoted from Petrov's book \cite{petrov}.

\begin{lemma}[\cite{petrov}, Lemma V.2.1, p. 109]\label{lem:esseen-ineq}
Let $X_1, \ldots, X_n$ be independent random variables, $\E X_j = 0$, $\E |X_j|^3 < \infty$ ($j=1,\ldots, n$). Denote $B_n = \sum_{j=1}^n \E X_j^2$, $L_n = B_n^{-3/2} \sum_{j=1}^n \E|X_j|^3$ and $S_n = B_n^{-1/2} \sum_{j=1}^n X_j$. Then
\[
  |\varphi_{S_n}(t) - e^{-t^2/2}| \le 16 L_n |t|^3 e^{-t^2/3}
\]
for $|t| \le \frac{1}{4L_n}$.
\end{lemma}
\begin{lemma}[\cite{petrov}, Lemma I.2.1, p. 10]\label{lem:upgrade-characteristic}
For any characteristic function $\varphi$,
\[
  1 - |\varphi(2t)|^2 \le 4(1 - |\varphi(t)|^2)
\]
holds for all $t \in \R$.
\end{lemma}

\begin{lemma}\label{lem:bounded-density}
Let $(S_n)$ be as in Theorem \ref{th:CLT-exp}. Let $T$ be independent of $(S_n)$ and assume that its characteristic function $\varphi_T$ is Lebesgue integrable.
Then for all $n\ge 1$, the density of $S_n+\frac{T}{n}$ is bounded by a number $C=C(1/\varepsilon,\delta, \nu_3,\|\varphi_T\|_1)$,
which is non-decreasing in each of its parameters.
\end{lemma}
\begin{proof}[Proof of Lemma \ref{lem:bounded-density}] 
Since $\varphi_{S_n+T/n}=\varphi_{S_n}\varphi_T(\cdot/n)$ is Lebesgue integrable,
the inversion formula ensures that the density of $S_n + \frac{1}{n} T$ at $x$ equals
\[ \begin{split}
  g_{S_n + \frac{1}{n}T}(x) &= \frac{1}{2\pi} \int_{-\infty}^\infty e^{-itx} \varphi_{S_n}(t) \varphi_T\big(t/n\big) \, dt \\
  &= \frac{1}{2\pi} \int_{-\infty}^\infty e^{-itx} e^{-t^2/2} \varphi_T\big(t/n\big) \, dt + \frac{1}{2\pi} \int_{-\infty}^\infty e^{-itx} (\varphi_{S_n}(t) - e^{-t^2/2}) \varphi_T\big(t/n\big) \, dt \\
  &\le \frac{1}{\sqrt{2\pi}} + \frac{1}{2\pi} \int_{-\infty}^\infty |\varphi_{S_n}(t) - e^{-t^2/2}| \, |\varphi_T\big(t/n\big)| \, dt.
\end{split} \]
To bound the last integral, we  apply Lemma~\ref{lem:esseen-ineq}
with $B_n = n$ and $L_n = \nu_3 n^{-1/2}$. We get
\[ \begin{split}
  g_{S_n + T/n}(x) &\le \frac{1}{\sqrt{2\pi}}
  + \frac{1}{2\pi} \int_{|t| \le \frac{\sqrt{n}}{4\nu_3}} \frac{16 \nu_3}{\sqrt{n}} |t|^3 e^{-t^2/3} \, dt
  + \frac{1}{2\pi} \int_{|t| > \frac{\sqrt{n}}{4\nu_3}} (|\varphi_{S_n}(t)| + e^{-t^2/2}) |\varphi_T\big(t/n\big)| \, dt \\
  &\le \frac{1}{\sqrt{2\pi}} + \frac{72 \nu_3}{\pi \sqrt{n}} +
  \frac{1}{2\pi} \underbrace{\int_{|t| > \frac{\sqrt{n}}{4\nu_3}} |\varphi_{S_n}(t)| \, |\varphi_T\big(t/n\big)| \, dt}_{I} +
  \frac{1}{2\pi} \int_{|t| > \frac{\sqrt{n}}{4\nu_3}} e^{-t^2/2} \, dt.
\end{split} \]
For integral (I) from the last line we use~\eqref{eq:cramer} which implies
\[
  |\varphi_{S_n}(t)| = \Big|\varphi_{{Y}_1}\big(t/\sqrt n\big)^n \Big|\le (1-\varepsilon)^n \quad \textup{for $|t| \ge \delta \sqrt{n}$.}
\]
However, $\delta$ might be larger than $\frac{1}{4\nu_3}$, i.e. $4\nu_3 \delta \ge 1$. If this is so, we use Lemma \ref{lem:upgrade-characteristic} on characteristic functions:
since~\eqref{eq:cramer} implies
\[
  1 - |\varphi_{{Y_i}}(t)|^2 \ge \varepsilon \quad \textup{for $|t| \ge \delta$,}
\]
Lemma \ref{lem:upgrade-characteristic} implies that for any non-negative integer $k$,
\[
  1 - |\varphi_{{Y_i}}(t)|^2 \ge 4^{-k} \varepsilon \quad \textup{for $|t| \ge 2^{-k} \delta$.}
\]
Taking $k = \lceil \log_2(4 \nu_3 \delta) \rceil)$ implies $2^{-k} \delta \le \frac{1}{4 \nu_3}$ and $4^{-k} \ge \frac{1}{(8 \nu_3 \delta)^2}$ and hence
\[
  |\varphi_{{Y_i}}(t)|^2 \le 1 - \frac{\varepsilon}{(8 \nu_3 \delta)^2} \quad \textup{for $|t| \ge \frac{1}{4 \nu_3}$.}
\]
In any case, we obtain that
\begin{equation}\label{eq:cramer-condition-2}
  |\varphi_{S_n}(t)| \le \left(1 - \frac{\varepsilon}{\max(1, (8 \nu_3 \delta)^2)}\right)^{n/2} \quad \textup{for $|t| \ge \frac{\sqrt{n}}{4\nu_3}$.}
\end{equation}
Using the above we estimate the integral (I) as follows. Using the rough estimate
\[
  (1-x)^m = e^{m\log(1-x)} \le e^{-mx}=\frac{1}{e^{mx}}\le \frac{1}{mx}\, ,
\]
valid for any $m > 0$ and $x \in (0,1)$
, we get
\[ \begin{split}
  I &\le \left(1 - \frac{\varepsilon}{\max(1, (8 \nu_3 \delta)^2)}\right)^{n/2} n \int_0^\infty |\varphi_T(u)| \, du \\
  &\le n \frac{2\max(1, (8 \nu_3 \delta)^2)}{n \varepsilon} \|\varphi_T\|_1
  \le 2 \|\varphi_T\|_1\frac{1 + (8\nu_3 \delta)^2}{\varepsilon}.
\end{split} \]
Finally we obtain that the density of $S_n + \frac{1}{n} T$ is bounded by $C_1 + C_2 \nu_3 + C_3 \|\varphi_T\|_1 \frac{1 + (\nu_3 \delta)^2}{\varepsilon}$ for some constants $C_1, C_2, C_3 > 0$.
\end{proof}

Denote by $\phi$ the density of the standard normal distribution on $\R$ and let $\Phi$ be its cumulative distribution function.
Our last two  preliminary statements are easy  consequences of the equality $ e^{\gamma t} \phi(t) = e^{\gamma^2/2} \phi(t - \gamma)$ satisfied by the Gaussian density
\begin{lemma}\label{lem:phi-shifted}
 Let $Z$ be a standard normal random variable. For any Borel set $A \subset \R$,
\begin{equation}\label{eq:normal-shifted}
  \E e^{\gamma Z} \II_{Z \in A} = e^{\gamma^2/2} \Pr(Z \in A - \gamma).
\end{equation}
\end{lemma}

\begin{lemma}\label{lem:exp-gaussian}
For any $s > 0$ and $\alpha \in \R$, and $\lambda$ such $\lambda s-\frac{\alpha }{s}>1$, it holds
\[
 \int_0^\infty \lambda e^{-\lambda x} \frac{1}{\sqrt{2\pi} s} e^{-\frac{(x - \alpha)^2}{2s^2}} \, dx= \frac{1}{\sqrt{2\pi} s} e^{-\frac{\alpha^2}{2s^2}} \left(1 + O\left(\frac{1+\frac{|\alpha|}{s}}{\lambda s- \frac{\alpha}{s}} \right) \right).
\]
In particular if $s$ and $\alpha$ stay bounded in the sense that $s\in [1/S,S]$, $|\alpha|\le A$ holds for some $A,S>0$, then for  $\lambda>2AS^{-2}+S^{-1}$, the last factor simplifies to $1+O_{A,S}\left(\frac{1}{\lambda}\right)$.
\end{lemma}
\begin{proof}
Using a standard Gaussian random variable $Z$, we  rewrite the left-hand side as
\[ \begin{split}
  \mathcal T:=\int_0^\infty \lambda e^{-\lambda x} \frac{1}{\sqrt{2\pi} s} e^{-\frac{(x - \alpha)^2}{2s^2}} \, dx
  &= \lambda \E e^{-\lambda (s Z + \alpha)} \II_{Z > -\frac{\alpha}{s}} = \lambda e^{-\lambda \alpha} e^{\lambda^2 s^2/2} \Pr\Big(Z > \lambda s - \frac{\alpha}{s}\Big) \\
  &= \lambda e^{-\frac{\alpha^2}{2s^2}} e^{(\lambda s - \frac{\alpha}{s})^2/2} \Big(1 - \Phi\big(\lambda s - \frac{\alpha}{s}\big)\Big) 
\end{split} \]
where the second equality follows from~\eqref{eq:normal-shifted}.
Next we use the classical bound, for $t>0$,
\[ \frac1t \ge \sqrt{2\pi} \, e^{t^2/2} \big( 1-\Phi(t)\big)\ge \frac{1}{\sqrt{t^2+2}},\]
which implies that for $t>1$, $\sqrt{2\pi} \, t e^{t^2/2} \big( 1-\Phi(t)\big)=1+O(1/t^2)$. When
$\lambda s-\frac{\alpha}{s}>1$ we obtain that 
\[
\begin{split}
\sqrt{2\pi} s \, e^{\frac{\alpha^2}{2s^2}}\, \mathcal T&= 
\frac{\lambda s}{ \lambda s-\frac{\alpha}{s}} \left(1+O\Big(\frac{1}{(\lambda s- \frac{\alpha}{s}) ^2} \Big) \right)
\\
&=\left( 1+\frac{\frac{\alpha}{s}}{\lambda s-\frac{\alpha}{s}}\right) \cdot \left(1+O\Big(\frac{1}{(\lambda s- \frac{\alpha}{s}) ^2} \Big) \right)= 1+O\left(\frac{1+\frac{|\alpha|}{s}}{\lambda s- \frac{\alpha}{s}} \right) .
\end{split}
\]
The case when $\alpha$ and $s$ are bounded readily follows.
\end{proof}

\section{Proof of Theorem \ref{th:CLT-exp}}
Our aim is to show that for any $\alpha \in \R$, $\mathcal I =\mathcal J\times (1+O(n^{-1/2}))$
where 
\[
  \mathcal I = \E e^{\ell \sqrt{n} S_n} \II_{\{S_n \le \alpha\}} \quad\mathrm{and}\quad
   \mathcal J= \frac{1}{\ell \sqrt{2\pi n}} e^{\ell \sqrt{n} \alpha} e^{-\alpha^2/2}.
\]
Let $Z$ be a standard Gaussian random variable, independent of the $Y_i$'s.
The first step is to introduce the modified quantity 
\[
  \mathcal I_2 = \E e^{\ell \sqrt{n} (S_n + n^{-1} Z)} \II_{\{S_n + n^{-1} Z \le \alpha\}},
\]
and to check that it is enough for our purpose to establish $\mathcal I_2 =\mathcal J\times (1+O(n^{-1/2}))$. In order to do so we estimate the difference between $\mathcal I$ and $\mathcal I_2$.

By the triangle inequality:
\[ \begin{split}
  |\mathcal I_2 - \mathcal I| &\le \E e^{\ell \sqrt{n} S_n} |e^{\ell n^{-1/2} Z} - 1| \II_{\{S_n \le \alpha\}}
  + \E e^{\ell \sqrt{n} (S_n + n^{-1} Z)} \big|\II_{\{  S_n + n^{-1} Z\le \alpha  \}} - \II_{\{  S_n \le \alpha \}} \big|\\
  &=\mathcal I_3+\mathcal I_4+\mathcal I_5,
\end{split} \]  
where 
\begin{align*}
 \mathcal{I}_3&= \E e^{\ell \sqrt{n} S_n} |e^{\ell n^{-1/2} Z} - 1| \II_{\{S_n \le \alpha\}}\\
 \mathcal{I}_4&= \E e^{\ell \sqrt{n} (S_n + n^{-1} Z)} \II_{\{  \alpha< S_n \le \alpha- n^{-1}Z \}}\\
 \mathcal{I}_5&= \E e^{\ell \sqrt{n} (S_n + n^{-1} Z)} \II_{\{  \alpha-n^{-1}Z< S_n \le \alpha \}}.
\end{align*}

By independence $  \mathcal I_3 = \mathcal I \cdot \E |e^{\ell n^{-1/2} Z} - 1| $. Next, we use 
that for $t\in[0,1]$,
\[ \E |e^{t Z} - 1| \le \sqrt{\E \big(e^{2tZ}-2e^{tZ}+1 \big) }= \sqrt{e^{2t^2}-2e^{t^2/2}+1}\le 3t.\]
Thus, under the hypothesis $n\ge 16 \ell^2$ we obtain that $\mathcal I_3\le  \frac{3\ell}{ \sqrt{n}}\mathcal I \le 3 \mathcal I/4$.

For the term $\mathcal I_4$, we introduce $T = U + U'$ where $U$ and $U'$ are independent random variables uniformly distributed in $(-1,1)$ and note that $\varphi_T(u)= (\sin(u)/u)^2$ is Lebesgue integrable. Since $|T|\le 2$ a.s., 
\[ \begin{split}
  \mathcal I_4 &\le e^{\ell \sqrt{n} \alpha} \int_0^\infty \Pr(\alpha < S_n \le \alpha + n^{-1} x) \phi(x) \, dx \\
  &\le e^{\ell \sqrt{n} \alpha} \int_0^\infty \Pr(\alpha - 2/n < S_n + T/n \le \alpha + (x+2)/n) \phi(x) \, dx.
\end{split} \]
By Lemma~\ref{lem:bounded-density}, $S_n + T/n$ has a density which is bounded by a constant, say $C > 0$. Then
\[ \begin{split}
  \mathcal I_4 & \le e^{\ell \sqrt{n} \alpha} \int_0^\infty C\frac{x+4}{n} \phi(x) \, dx = \frac{C}{n} e^{\ell \sqrt{n} \alpha} (\pi^{-1/2} + 2)\\
  &=  \frac{C}{\sqrt n} \cdot  \ell \sqrt{2\pi } \, e^{\alpha^2/2} \mathcal J \cdot (\pi^{-1/2} + 2) = \mathcal J \cdot O(n^{-1/2}).
\end{split}
\]

The term $\mathcal I_5$ is estimated in a similar way:
\[ \begin{split}
  \mathcal I_5 &\le e^{\ell \sqrt{n} \alpha} \int_0^\infty e^{\ell xn^{-1/2}}\Pr(\alpha-x/n < S_n \le \alpha) \phi(x) \, dx \\
  &\le e^{\ell \sqrt{n} \alpha} \int_0^\infty e^{\ell xn^{-1/2}}\Pr(\alpha-(x+2)/n < S_n+T/n \le \alpha+2/n) \phi(x) \, dx \\
   &\le e^{\ell \sqrt{n} \alpha} \int_0^\infty e^{\ell x}C \frac{x+4}{n} \phi(x) \, dx = \mathcal J \cdot O(n^{-1/2}). 
\end{split} \]
This concludes the first step of the proof, which guarantees that for $n\ge 16 \ell^2$
\begin{equation}\label{pf:step1}
|\mathcal I_2-\mathcal{I}|\le \frac{3\ell}{\sqrt n} \, \mathcal I+ O\Big(\frac{1}{\sqrt n}\Big)\mathcal J.	
\end{equation} 

Our next task is to prove that $\mathcal I_2 =\mathcal J\times (1+O(n^{-1/2}))$. We use the Fourier transform approach. 
It relies on the  Parseval formula, which  ensures that whenever random variables $V$ and $W$ have square integrable densities $g_V$ and $g_W$, their characteristic functions are also square integrable and the following relation holds:
\begin{equation}\label{eq:parseval-formula}
  \int_{-\infty}^\infty g_V(x) g_W(x) \, dx = \frac{1}{2\pi} \int_{-\infty}^\infty \varphi_V(t) \overline{\varphi_W(t)} \, dt.
\end{equation}

 Given $n$, set $W = \alpha - (S_n + \frac{1}{n}Z)$. Then
\begin{align*}
\mathcal I_2 &=\E e^{\ell \sqrt{n} (S_n + n^{-1} Z)} \II_{\{S_n + n^{-1} Z \le \alpha\}}=  e^{\ell \sqrt{n} \alpha} \E e^{-\ell \sqrt{n} W} \II_{W \ge 0}\\
  &= \frac{e^{\ell \sqrt{n} \alpha}}{\ell \sqrt{n}} \int_0^\infty \ell \sqrt{n} e^{-\ell \sqrt{n} x} \, d\mathbb P_W(x).
\end{align*}
Let $V$ a random variable having exponential distribution with parameter $\ell \sqrt{n}$. We have proved that 
\[ \widetilde{\mathcal I}_2:= \ell\sqrt{n}e^{-\ell \sqrt{n} \alpha} \mathcal{I}_2=\int g_V(x) \, d\mathbb P_W(x).  \]
Observe that our goal is to establish that $\widetilde{\mathcal I}_2=\frac{1}{\sqrt{2\pi}}e^{-\alpha^2/2}(1+O(n^{-1/2}))$. 

Since $\mathbb P_W$ is given by the convolution of a probability measure and of the bounded density of $Z/n$, it is absolutely continuous with bounded  (and thus square-integrable) density.
Hence, we may apply the Parseval formula ~\eqref{eq:parseval-formula}  to $V$ and $W$. Since 
$\varphi_W(t) = e^{i\alpha t} \overline{\varphi_{S_n}(t)} e^{-t^2/(2n^2)}$,  we obtain

 \[ \widetilde{\mathcal I}_2 = \frac{1}{2\pi} \int_{-\infty}^\infty \frac{1}{1-\frac{it}{\ell \sqrt{n}}} e^{-i\alpha t} \varphi_{S_n}(t) e^{-t^2/(2n^2)} \, dt =\frac{\mathcal M+ \mathcal E }{2\pi},
\]  
where
 \begin{align*} 
  \mathcal M &=  \int_{-\infty}^\infty \frac{e^{-i\alpha t}}{1-\frac{it}{\ell \sqrt{n}}} e^{-t^2/2} e^{-t^2/(2n^2)} \, dt \\
  \mathcal E &= \int_{-\infty}^\infty \frac{e^{-i\alpha t}}{1-\frac{it}{\ell \sqrt{n}}} (\varphi_{S_n}(t) - e^{-t^2/2}) e^{-t^2/(2n^2)} \, dt.
\end{align*}
Applying Parseval's formula as before, but replacing $S_n$ with and independent standard Gaussian variable $G$ yields $\mathcal{M}/(2\pi)=\int g_Vd\mathbb{P}_{\widetilde W}$ where
$\widetilde W=\alpha-(G+Z/n)$ has $\mathcal{N}(\alpha, 1+n^{-2})$ distribution. Therefore
\[ \frac{\mathcal M}{2\pi}= \int_{-\infty}^\infty \ell \sqrt{n} e^{-\ell \sqrt{n}} \frac{e^{-\frac{(x-\alpha)^2}{2(1+n^{-2})}}}{\sqrt{2\pi (1+n^{-2})}} dx.\]
Lemma~\ref{lem:exp-gaussian} with $\lambda:=\ell \sqrt{n}$ and $s^2:=1 + n^{-2}$ yields, provided $\ell \sqrt{n}\ge 2|\alpha|+1$, 
\[ \begin{split}
  \frac{\mathcal M}{2\pi} &= \frac{1}{\sqrt{2\pi(1 + n^{-2})}} e^{-\frac{\alpha^2}{2(1+n^{-2})}}(1 + O(n^{-1/2}))  \\
  &= \frac{1}{\sqrt{2\pi}} e^{-\frac{\alpha^2}{2}}(1 + O(n^{-1/2})) .
\end{split} \]
It remains to bound the error term:
\[ \begin{split}
  |\mathcal E| &= \left| \int_{-\infty}^\infty \frac{e^{-i\alpha t}}{1-\frac{it}{\ell \sqrt{n}}} (\varphi_{S_n}(t) - e^{-t^2/2}) e^{-t^2/(2n^2)} \, dt  \right| \\
  &\le \int_{-\infty}^\infty \big|\varphi_{S_n}(t) - e^{-t^2/2}\big| \, e^{-t^2/(2n^2)} \, dt \\
  &\le \int_{|t| \le \sqrt{n}/(4\nu_3)} 16 \nu_3 n^{-1/2} |t|^3 e^{-t^2/3} \, dt \\
  &+ \int_{|t| > \sqrt{n}/(4\nu_3)} \big|\varphi_{S_n}(t)\big|\,  e^{-t^2/(2n^2)} \, dt + \int_{|t| > \sqrt{n}/(4\nu_3)} e^{-t^2/2} \, dt \\
  &\le C\nu_3 n^{-1/2} + I + II,
\end{split} \]
where the second inequality follows from Lemma~\ref{lem:esseen-ineq}. The estimate of the term $II$ is immediate:
\[
  II \le 2 e^{-n/(32 \nu_3^2)}.
\]
In order to estimate $I$, we use~\eqref{eq:cramer-condition-2} and a variant of its previous
application using the bound $(1-x)^m\le 1/e^{mx}\le 2/(mx)^2$ for $x\in (0,1)$:
\[
  I \le \left(1 - \frac{\varepsilon}{\max(1, (8 \nu_3 \delta)^2)}\right)^{n/2} n \sqrt{2\pi} = O_{\nu_3, \frac1\varepsilon, \delta}(n^{-1/2}).
\]
Hence $\mathcal E =O(n^{-1/2})=e^{-\alpha^2/2}O(e^{\alpha^2/2}n^{-1/2})=e^{-\alpha^2/2}O_{|\alpha|}(n^{-1/2})$.
This ends the proof of the second step, asserting $\mathcal I_2 =\mathcal J\times (1+O(n^{-1/2}))$. 
Combining the latter with \eqref{pf:step1} yields the claim of the theorem.


\section{Application to spectral gaps}\label{sec:gap}
 
 Our volume asymptotics for Orlicz balls allow to complement a result of Kolesnikov and Milman \cite{kolesnikov-milman}
 about a famous conjecture by Kannan, Lov\'asz and Simonovits, which predicts the approximate value of the Poincar\'e constants of convex bodies (a.k.a. inverse spectral gap of the Neumann Laplacian). More precisely if $\mu$ is a probability measure on some Euclidean space, one denotes by $C_P(\mu)$ (resp. $C_P^{Lin}(\mu)$)
the smallest constant $C$ such that for all locally Lipschitz (resp. linear) functions $f$, it holds
\[ \mathrm{Var}_\mu(f)\le C \int |\nabla f|^2 d\mu.\]
Obviously  $C_P^{Lin}(\mu)\le C_P(\mu)$, and the KLS conjecture predicts the existence of a universal constant $c$ such that for any dimension $n$ and any convex body $K\subset \mathbb R^n$, 
\[C_P(\lambda_K)\le c\,  C_P^{Lin}(\lambda_K),\]
where $\lambda_K$ stands for the uniform probability measure on $K$. The conjecture turned out to be central in the understanding in high-dimension volume distributions of convex sets.
We refer to e.g. \cite{abBOOK, greekBOOK,kolesnikov-milman,lee-vempala,chen} for more background and references, and to \cite{kl} for a recent breakthrough. Kolesnikov and Milman have verified the conjecture for some Orlicz balls. We state next a simplified version of their full result on generalized Orlicz balls. Part of the simplification is unessential, as it amounts to reduce by dilation and translations to a convenient setting. A more significant simplification, compared to their work, is that we consider balls where all coordinates play the same role. 

\begin{theorem}[\cite{kolesnikov-milman}]
Let $V:\R\to \R^+$ be a convex function with $V(0)=0$ and such that $d\mu(x)=e^{-V(x)}dx$
is a probability measure.  We also assume that the function $x\mapsto x V'(x)$, defined almost everywhere, belongs to the space $L^2(\mu)$. For each dimension $n\ge 1$, let 
\[ \mathrm{Level}_n(V):=\left\{ E\ge 0; \; e^{-E} \mathrm{Vol}_n\big(B^n_{V/E}\big) \ge \frac{1}{e}\, \frac{n^n e^{-n}}{n!}\right\}.\]
Then there exists a constant $c$, which depends only on $V$ (through $\|xV'(X)\|_{L^2(\mu)}$)
such that for all $E\in \mathrm{Level}_n(V)$,
\[ C_P(\lambda_{B^n_{V/E}})\le c\,  C_P^{Lin}(\lambda_{B^n_{V/E}}).\]
Moreover, $\mathrm{Level}_n(V)$ is an interval of length at most $e \frac{n! e^{n}}{n^n}= e\sqrt{2\pi n}(1+o(1))$ as $n\to \infty$, and 
\[ 1+ n \int_{\R} V(x) e^{-V(x)} dx\in \mathrm{Level}_n(V).\]
\end{theorem}
We can prove more about the set $ \mathrm{Level}_n(V)$ and in particular we show that its length is of order $\sqrt n$:
\begin{prop}
Let $V:\R\to \R^+$ be a Young function 
such that $d\mu(x)=e^{-V(x)}dx$
is a probability measure. Let $m_1=\int V e^{-V}$ be the average of $V$ with respect to $\mu$, and $\sigma_1^2$ its variance. 
For every $\varepsilon\in (0,1)$ there exists an integer $n_0=n_0(V,\varepsilon)$ depending on $V$ such that 
for all $n\ge n_0$,
\[    \left[m_1n-\sigma_1 (1-\varepsilon) \sqrt{2n}\, ;\, m_1n+\sigma_1 (1-\varepsilon) \sqrt{2n}\right] \subset  \mathrm{Level}_n(V).\]
\end{prop}
\begin{proof}
We apply Theorem \ref{th:volume-asymptotic}, with $\Psi=V$ and $\lambda=1$. With the notation of the theorem $\mu=\mu_1$ and $Z_1=\int e^{-V}=1$.  We choose $E$ of the following form: $E=m_1n+\alpha \sigma_1 \sqrt{n}$ with $|\alpha|\le (1-\varepsilon)\sqrt 2$.
The theorem ensures that 
\[ \Vol\big(B^n_{V/E}\big)=\frac{e^E}{\sigma_1\sqrt{2\pi n}} e^{-\alpha^2/2} \left(1+O\Big(\frac{1}{\sqrt{n}}\Big)\right),\]
where the $O(n^{-1/2})$ is uniform in $\alpha\in [-(1-\varepsilon)\sqrt 2,(1-\varepsilon)\sqrt 2]$. A sharp inequality due to Nguyen and Wang ensures that $\sigma_1^2=\Var_{e^{-V}}(V)\le 1$  (see \cite{nguyen-phd,wang-phd}, \cite{nguyen} and for a short proof \cite{fradelizi-m-w}). Therefore
\[ e^{-E}\Vol\big(B^n_{V/E}\big)\ge \frac{1}{\sqrt{2\pi n}} e^{-(1-\varepsilon)^2} \left(1+O\Big(\frac{1}{\sqrt{n}}\Big)\right),\]
whereas
\[\frac{1}{e} \frac{n^n e^{-n}}{n!}=\frac{e^{-1}}{\sqrt{2\pi n}} (1+o(1)) .\]
Hence for $n$ large enough and for all $\alpha$ in the above interval $e^{-E}\Vol\big(B^n_{V/E}\big)\ge \frac{1}{e} \frac{n^n e^{-n}}{n!}$.
\end{proof}

\begin{corollary}
Let $V:\R\to \R^+$ be a Young function 
such that $d\mu(x)=e^{-V(x)}dx$
is a probability measure. Let   $m_1$ and $\sigma_1^2$ denote the average and the variance of $V$ with respect to $\mu$.  We also assume that the function $x\mapsto x V'(x)$ belongs to the space $L^2(\mu)$. Let $\varepsilon\in (0,1)$. Then there exists $c=c(V,\varepsilon)$ such 
that for all $n\ge 1$ and all $E\in \left[m_1n-\sigma_1 (1-\varepsilon) \sqrt{2n}\, ;\, m_1n+\sigma_1 (1-\varepsilon) \sqrt{2n}\right] $,
\[ C_P(\lambda_{B^n_{V/E}})\le c\,  C_P^{Lin}(\lambda_{B^n_{V/E}}).\]
\end{corollary}
\begin{proof}
Combining the later proposition and theorem yields the result for $n\ge n_0(V,\varepsilon)$.
 In order to deal with smaller dimensions, we simply apply known dimension dependent bounds: e.g. Kannan, Lov\'asz and Simonovits \cite{kls} proved that 
$C_P(\lambda_K)\le \kappa n C_P^{Lin}(\lambda_K)$ for all convex bodies $K$ in $\R^n$, with 
$\kappa$ a universal constant. 
\end{proof}

\section{Asymptotic independence of coordinates}\label{sec:indep}
A classical observation, going back to Maxwell, but also attributed to Borel and to Poincaré, states that 
for a fixed $k$, the law of the first $k$ coordinates of a uniform random vector on the Euclidean sphere of $\mathbb R^n$, centered at the origin and of radius $\sqrt{n}$, tends to the law of $k$ independent standard Gaussian random variables as $n$ tends to infinity. Quantitative versions of this asymptotic independence property where given by Diaconis and Freedman \cite{diaconis-freedman}, as well as a similar result for the unit sphere of the 
$\ell_1$-norm, involving exponential variables in the limit. Extensions to random vectors distributed according to the cone measure on the surface of the unit ball $B_p^n$ were given by Rachev and R\"uschendorf \cite{rachev-r}, while Mogul'ski\u{\i}     \cite{mogulskii} dealt with the case of the normalized surface measure. Explicit calculations, or the probabilistic representation put forward in \cite{bgmn}, easily yield asymptotic independence results for the first $k$ coordinates of a uniform vector on the set $B_p^n$ itself, when $k$ is fixed and $n$ tend to infinity. 

In this section we study marginals of a random vector $\xi^{(n)}$ uniformly distributed on $B_{\Psi/E_n}^n$, where $E_n$ and $n$ tend to $\infty$.

Let us start with the simple case when $E_n=mn$ for some $m>0$, which can be written as $m=m_\lambda$ for some $\lambda>0$. Let $k\ge 1$ be a fixed integer, then the density at $(x_1,\ldots,x_k)\in \R^k$ of the first $k$ coordinates  $(\xi^{(n)}_1, \ldots, \xi^{(n)}_{k})$ is equal to 
\[ \frac{\Vol_{n-k} \big(B^n_{\Phi/E_n}\cap \{y\in \mathbb R^n;\; y_i=x_i, \, \forall i\le k \}\big) }{\Vol\big(B^n_{\Phi/E_n}\big)}=\frac{\Vol\big(B_{\Psi/{(E_n - \sum_{i=1}^{k} \Psi(x_i))}}^{n-k}\big)}{\Vol\big(B_{\Psi/E_n}^n\big)} \] 
We apply Corollary \ref{cor:volume-asymptotic} twice: once for the denominator, and once for the numerator after writing 
\[m_\lambda n-\sum_{i\le k}\Psi(x_i)= m_\lambda (n-k)+  \frac{m_\lambda k-\sum_{i\le k}\Psi(x_i) }{\sqrt{n-k}} \sqrt{n-k}.\]
We obtain that the above ratio is equivalent to 
\[  \frac{Z_\lambda^{n-k} e^{\lambda (E_n-\sum_{i\le k} \Psi(x_i))}}{\lambda \sigma_\lambda \sqrt{2\pi (n-k)}} \cdot \frac{\lambda \sigma_\lambda \sqrt{2\pi n}}{Z_\lambda^n e^{\lambda E_n}}\sim \frac{e^{-\lambda\sum_{i=1}^k \Psi(x_i)}}{Z_\lambda^k}\cdot
\]
Thus we have proved the convergence in distribution of $(\xi^{(n)}_1, \ldots, \xi^{(n)}_{k})$
to $\mu_\lambda^{\otimes k}$ as $n$ tends to infinity. In other words the first $k$ coordinates
of $\xi^{(n)}$ are asymptotically i.i.d. of law $\mu_\lambda$. 
This is true for more general balls and  for a number of coordinates going also to infinity:
\begin{theorem}\label{th:marginals}
Let $E_n = m_\lambda n+\alpha_n \sigma_\lambda \sqrt n$, where $(\alpha_n)_{n\ge 1}$ is bounded.  Let the random vector $\xi^{(n)}$ be uniformly distributed on $B_{\Psi/E_n}^n$. For any $k_n = o(\sqrt{n})$, 
\[
 \lim_{n \to \infty}  d_{TV}\big((\xi^{(n)}_1, \ldots, \xi^{(n)}_{k_n}), \mu_\lambda^{\otimes k_n}\big) = 0.
\]
\end{theorem}
\begin{proof}
Below, we simply  write 
$\xi_i$ for $\xi^{(n)}_i$. Recall that $(X_i)$ are i.i.d. r.v.'s with the distribution $\mu_\lambda$. Set $t_n := n^{1/4}k_n^{1/2}$ so that $t_n = o(\sqrt{n})$ and $k_n = o(t_n)$. The total variation distance between the law of $(\xi^{(n)}_1, \ldots, \xi^{(n)}_{k_n})$ and $\mu_\lambda^{\otimes k_n}$ is
\begin{align}
  \nonumber
  & \int_{\R^{k_n}} \left| \frac{1}{\Vol(B_{\Psi/E_n}^n)} \int_{\R^{n-k_n}} \II_{\{(x, y) \in B_{\Psi/E_n}^n\}} \, dy - \frac{1}{Z_\lambda^{k_n}} e^{-\lambda(\Psi(x_1) + \cdots + \Psi(x_{k_n}))} \right| \, dx \\
  \nonumber
  \le& \int_{B_{\Psi/t_n}^{k_n}} \left| \frac{\Vol\big(B_{\Psi/{(E_n - \sum_{i=1}^{k_n} \Psi(x_i))}}^{n-k_n}\big)}{\Vol(B_{\Psi/E_n}^n)} - \frac{1}{Z_\lambda^{k_n}} e^{-\lambda(\Psi(x_1) + \cdots + \Psi(x_{k_n}))} \right| \, dx \\
  \nonumber
  &+ \Pr\big((\xi_1, \ldots, \xi_{k_n}) \not\in B_{\Psi/t_n}^{k_n}\big) 
   + \Pr\big((X_1, \ldots, X_{k_n}) \not\in B_{\Psi/t_n}^{k_n}\big) \\
   \label{eq:d_TV-term}
  =& \int_0^{t_n} \left| \frac{\Vol\big(B_{\Psi/{(E_n - t)}}^{n-k_n}\big)}{\Vol(B_{\Psi/E_n}^n)} - \frac{e^{-\lambda t}}{Z_\lambda^{k_n}} \right| \frac{d}{dt} \Vol(B_{\Psi/t}^{k_n}) \, dt \\
  \nonumber
  &+ \Pr\left(\sum_{i=1}^{k_n} \Psi(\xi_i) > t_n\right) + \Pr\left(\sum_{i=1}^{k_n} \Psi(X_i) > t_n\right).
\end{align}
By Markov's inequality,
\[
\Pr\left(\sum_{i=1}^{k_n} \Psi(X_i) > t_n\right) \le \frac{\E\big(\sum_{i=1}^{k_n} \Psi(X_i)\big)}{t_n} =\frac{k_n  m_\lambda}{t_n} = o(1)
\]
Similarly, and since by definition $\sum_{i=1}^n \Psi(\xi_i)\le E_n$ and the  $\xi_i$'s
are exchangeable
\[
  \Pr\left(\sum_{i=1}^{k_n} \Psi(\xi_i) > t_n\right) \le \frac{\E\big(\sum_{i=1}^{k_n} \Psi(\xi_i)\big)}{t_n}
  \le \frac{k_n E_n}{n t_n} = \frac{k_n m_\lambda}{t_n} = o(1).
\]
In order to estimate~\eqref{eq:d_TV-term}, we use Theorem~\ref{th:volume-asymptotic}. Since $k_n = o(\sqrt{n})$ and $t_n = o(\sqrt{n})$, we know that $E_n-t=m_\lambda (n-k_n)+\beta_n \sigma_\lambda \sqrt{n-k_n} $, where 
\[
\beta_n:= \alpha_n \sqrt{\frac{n}{n-k_n}}+ \frac{m_\lambda k_n-t}{\sigma_\lambda \sqrt{n-k_n}} 
\]
is a bounded sequence such that $\beta_n-\alpha_n=o(1)$, both properties holding
uniformly in $t \in [0, t_n]$. Therefore, Theorem~\ref{th:volume-asymptotic} applied to $B_{\Psi/(E_n-t)}^{n-k_n}$ gives
\[
  \Vol\big(B_{\Psi/{(E_n - t)}}^{n-k_n}\big) = \frac{Z_\lambda^{n-k_n} e^{\lambda (E_n - t)}}{\lambda \sigma_\lambda \sqrt{2\pi(n-k_n)}}  e^{-\beta_n^2/2}(1+o(1))
\]
uniformly in $t \in [0, t_n]$. On the other hand, Theorem~\ref{th:volume-asymptotic} applied to  $B_{\Psi/E_n}^n$ yields
\[
  \Vol\big(B_{\Psi/{E_n }}^{n}\big) = \frac{Z_\lambda^{n} e^{\lambda E_n }}{\lambda \sigma_\lambda \sqrt{2\pi n}}  e^{-\alpha_n^2/2}(1+o(1)).
\]
Combining the above two asymptotic expansions, we obtain
\[
  \frac{\Vol\big(B_{\Psi/{(E_n - t)}}^{n-k_n}\big)}{\Vol(B_{\Psi/E_n}^n)} = \frac{e^{-\lambda t}}{Z_\lambda^{k_n}} (1+o(1))
\]
uniformly in $t \in [0, t_n]$. Therefore the term~\eqref{eq:d_TV-term} equals
\[
  o(1) \int_0^{t_n} \frac{e^{-\lambda t}}{Z_\lambda^{k_n}} \frac{d}{dt} \Vol(B_{\Psi/t}^{k_n}) \, dt
  = o(1) \cdot \Pr\left(\sum_{i=1}^{k_n} \Psi(X_i) \le t_n\right) = o(1).
\]
\end{proof}

The next result gives the asymptotic distribution of a sort of distance to the boundary for high-dimensional Orlicz balls.
\begin{theorem}\label{th:boundary-distance}
Let $E_n = m_\lambda n+\alpha_n \sigma_\lambda \sqrt n$, where $(\alpha_n)_{n\ge 1}$ is bounded.  Let the random vector $\xi^{(n)}$ be uniformly distributed on $B_{\Psi/E_n}^n$. 
Then the following convergence in distribution occurs as $n$ goes to infinity:
\[ 
\lambda\cdot  \Big( E_n-\sum_{i=1}^n \Phi\big(\xi_i^{(n)}\big) \Big) \longrightarrow \mathcal{E}xp(1).
\]
\end{theorem}
\begin{proof}
Let $S_n:=E_n-\sum_{i=1}^n \Phi\big(\xi_i^{(n)}\big)\ge 0$.  For $t\ge 0$,
\[
\Pr(S_n\ge t) = \Pr \left(\sum_{i=1}^n \Phi\big(\xi_i^{(n)}\big)\le E_n-t \right)
=\frac{\Vol\big(B^n_{\Psi/(E_n-t)}\big)}{\Vol\big(B^n_{\Psi/E_n}\big)} \cdot
\]
As before, Theorem~\ref{th:volume-asymptotic} applied to  $B_{\Psi/E_n}^n$ yields
\[
  \Vol\big(B_{\Psi/{E_n }}^{n}\big) \sim \frac{Z_\lambda^{n} e^{\lambda E_n }}{\lambda \sigma_\lambda \sqrt{2\pi n}}  e^{-\alpha_n^2/2},
\]
whereas applied to $B_{\Psi/E_{n-t}}^n$ it gives
\[
  \Vol\big(B_{\Psi/{E_n-t }}^{n}\big) \sim \frac{Z_\lambda^{n} e^{\lambda (E_n-t) }}{\lambda \sigma_\lambda \sqrt{2\pi n}}  e^{-\left(\alpha_n-\frac{t}{\sigma_\lambda \sqrt{n}}\right)^2/2}.
\]
Taking the quotient gives $\lim_n \Pr(S_n\ge t) =e^{-\lambda t}.$	
\end{proof}

\section{Integrability of linear functionals}\label{sec:linear}
Linear functionals of uniform random vectors on convex bodies are well studied quantities. Their density function, known as the parallel section function, measures the volume of hyperplane sections in a given direction. We refer e.g. to the book \cite{greekBOOK}, and in particular to its sections 2.4 and 8.2 about the $\psi_1$ and $\psi_2$ properties, which describe uniform integrability features (exponential integrability for $\psi_1$, Gaussian type integrability for $\psi_2$). They can be expressed by upper bounds on the Laplace transform.

In this section, we deal with even Young functions $\Psi$, so that the corresponding sets
$B^n_\Psi$ are origin-symmetric, and actually unconditional. The forthcoming study is valid for any dimension, without taking limits, so we consider the dimension $n$ fixed, and write $\xi=(\xi_1,\ldots,\xi_n)$
for a uniform random vector on $B^n_\Psi$. We show that the arguments of \cite{barthe-koldobsky} for $\ell_p^n$ unit balls extend to Orlicz balls. 

\begin{lemma}\label{lem:laplace}
Let $a\in \R^n$, and $\xi$ be uniform on $B^n_\Psi$, then
\[ \E e^{\langle a, \xi\rangle}\le \prod_{i=1}^n \E e^{a_i \xi_1}.\]
\end{lemma}
\begin{proof}
 Let $\varepsilon_1,\ldots,\varepsilon_n$ be i.i.d. random variables with $\Pr(\varepsilon_i=1)=
 \Pr(\varepsilon_i=-1)=\frac12$, and independent of $\xi$. Then by symmetry of $\Psi$,
 $(\varepsilon_1\xi_1,\ldots,\varepsilon_n\xi_n)$ has the same distribution as $\xi$. Hence,
 \[ \E e^{\langle a, \xi\rangle}=\E \prod_{i=1}^n e^{a_i\varepsilon_i\xi_i}=\E\left( \E\Big( \prod_{i=1}^n e^{a_i\varepsilon_i\xi_i}\, \Big| \, \xi\Big)\right)=\E \prod_{i=1}^n  \cosh(a_i \xi_i).\]
 Next by the subindependence property of coordinates, due Pilipczuk and Wojtaszczyk \cite{p-w}, and using the symmetry again as well as exchangeability:
 \[ \E e^{\langle a, \xi\rangle} \le \prod_{i=1}^n\E \cosh(a_i \xi_i)= \prod_{i=1}^n\E e^{a_i \xi_i}
 =\prod_{i=1}^n\E e^{a_i \xi_1}.\]
\end{proof}
The above lemma shows that the Laplace transform of any linear functional $\langle a,\xi\rangle$ can be upper estimated using the Laplace transform of the first coordinate $\xi_1$. Therefore it is natural to study the law of $\xi_1$. For $t\in\R$ consider the section of $B_\Psi^n$:
\[ S(t):=\{ y\in \R^{n-1};\; (t,y)\in B_\Psi^n\}.\]
and $f(t):=\Vol_{n-1}(S(t))$. Then $\Pr_{\xi_1} (dt)=f(t)dt/\Vol_n(B^n_\Psi)$. By the Brunn principle, $f$ is a log-concave function. It is also even by symmetry of the ball, therefore it is non-increasing on $\R^+$. We observe that a slightly stronger property holds:

\begin{lemma}\label{lem:section}
Let $\Psi$ be an even Young function and $f(t)=\Vol_n\big(\{ y\in \R^{n-1};\; (t,y)\in B_\Psi^n\}
\big)$. Then the function $\log f\circ \Psi^{-1}$ is concave and non-increasing on $\R^+$.
Here $\Psi^{-1}$ is the reciprocal function of the restriction of $\Psi$ to $\R^+$.
 \end{lemma}
\begin{proof}
Let $t,u\ge 0$. Let $a\in S(t)$ and $b\in S(u)$. Then by definition
 \[
\Psi(t)+\sum_{i=1}^{n-1} \Psi(a_i)\le 1 \quad\mathrm{and}\quad
\Psi(u)+\sum_{i=1}^{n-1} \Psi(b_i)\le 1. 
 \]
 Averaging these two inequalities and using the convexity of $\Psi$, we get for any $\theta\in (0,1)$:
 \begin{equation} \label{eq:section}
 (1-\theta)\Psi(t)+\theta \Psi(u) +\sum_{i=1}^{n-1} \Psi\left((1-\theta)a_i+\theta b_i\right)\le 1.
\end{equation}  
 This can be rewritten as 
   \[ (1-\theta)a+\theta b \in S\left( \Psi^{-1}\left( (1-\theta)\Psi(t)+\theta\Psi(u)\right) \right). \]
 Hence we have shown that 
  \[ (1-\theta)S(t)+\theta S(u) \subset S\left( \Psi^{-1}\left( (1-\theta)\Psi(t)+\theta \Psi(u)\right) \right),\]
  and by the Brunn-Minkowski inequality, in multiplicative form
  \[ f(t)^{1-\theta} f(u)^\theta\le  f\left( \Psi^{-1}\left( (1-\theta)\Psi(t)+\theta \Psi(u)\right) \right).\]
  Note that if in \eqref{eq:section} we had used convexity in the form 
  $\Psi((1-\theta)t+\theta u)\le  (1-\theta)\Psi(t)+\theta \Psi(u)$, then we would have derived the Brunn principle from the Brunn-Minkowski inequality. 
\end{proof}

The next result shows that $\Psi$ is more convex than the square function, the corresponding 
Orlicz balls enjoy the $\psi_2$ property. This applies in particular to $B_p^n$ for $p\ge 2$, a case which was treated in \cite{barthe-koldobsky}.

\begin{theorem}
Let $\Psi$ be an even Young function, such that $t>0\mapsto \Psi(\sqrt{t})$ is convex. 
Let $\xi$ be a uniform random vector on $B^n_\Psi$.
Then for all $a\in \R^n$,
\[  \E e^{\langle a, \xi\rangle}\le \left( \E e^{\frac{|a|}{\sqrt n}\xi_1}\right)^n\le e^{\frac12\E \big(\langle a,\xi\rangle^2\big)}.\]
\end{theorem}
\begin{proof}
Let $L_X(t)=\E e^{tX}$ denote the Laplace transform of a real valued random variable. Then with the notation of Lemma \ref{lem:section},
\[ L_{\xi_1}(t)=\int e^{tu} f(u) \frac{du}{\Vol(B^n_\Psi)}.\]
Lemma \ref{lem:section} ensures that there exists a concave function $c$ such that
for all $u\ge 0$, $ \log f(u)=c(\Psi(u))$. Note that $c$ is also non-increasing on $\R^+$ since
the section function $f$ is. Hence 
\[ u\ge 0 \mapsto \log f(\sqrt{u})=c(\Psi(\sqrt u))\]
is concave. Theorem 12 of \cite{barthe-koldobsky} ensures that  $t\ge 0\mapsto \int_{\R} e^{u\sqrt t} f(u) \,du$ is log-concave. In other words,
 \[ t\ge 0\mapsto \log L_{\xi_1}(\sqrt t)\]
 is concave. 
 
 From Lemma \ref{lem:laplace}, using symmetry and the above concavity property
 \[ \E e^{\langle a,\xi\rangle} \le \prod_{i=1}^n L_{\xi_1}(a_i)= \prod_{i=1}^n L_{\xi_1}\left(\sqrt{a_i^2}\right)\le \left( L_{\xi_1}\left( \sqrt{\frac1n \sum_i a_i^2}\right)\right)^n
=L_{\xi_1}\left(\frac{|a|}{\sqrt n}\right)^n. 
 \]
 To conclude we need the bound $L_{\xi_1}(t)\le e^{t^2\E(\xi_1^2)/2}$ (it follows from the fact
 that $t\ge 0\mapsto \log L_{\xi_1}(\sqrt t)$ is concave, hence upper bounded by its tangent application at 0, which is easily seen to be $t\E(\xi_1^2)/2$).  We obtain 
  \[ \E e^{\langle a,\xi\rangle} \le e^{\frac12 |a|^2\E(\xi_1^2)},\]
 and we conclude using the symmetries of $\xi$ since
 \[ \E \big(\langle a,\xi\rangle^2\big)=\sum_i a_i^2 \E (\xi_i^2)+ \sum_{i\neq j} a_ia_j \E(\xi_i\xi_j)= |a|^2 \E(\xi_1^2). \] 
\end{proof}

{\bf Acknowledgements:} We are grateful to Emanuel Milman and Reda Chhaibi for useful discussions on related topics. We also thank Joscha Prochno for communicating his recent work to us.


\bigskip
\noindent Institut de Math\'ematiques de Toulouse, UMR 5219\\
Universit\'e de Toulouse \& CNRS
\\ UPS, F-31062 Toulouse Cedex 09, France. 

\hfill \verb"barthe@math.univ-toulouse.fr"

\hfill \verb"pwolff@mimuw.edu.pl"
\end{document}